\documentclass{article}
\usepackage[english]{babel}

\usepackage[letterpaper,top=3cm,bottom=3cm,left=3cm,right=3cm,marginparwidth
=1.75cm]{geometry}
\usepackage{indentfirst}
\usepackage{amsmath,amssymb,amsbsy}
\usepackage{mathrsfs}
\usepackage{amsthm}
\usepackage{graphicx}
\usepackage[colorlinks=true, allcolors=blue]{hyperref}
\usepackage{amsfonts}
\usepackage{authblk}
\usepackage{fancyhdr}
\DeclareMathOperator*{\Res}{Res}

\theoremstyle{definition}

\theoremstyle{plain}
\newtheorem{Th}{Theorem}[section]
\newtheorem{Lem}{Lemma}[section]
\theoremstyle{remark}

\numberwithin{equation}{section}
\title{On certain kernel functions and shifted convolution sums of Hecke eigenvalues}
\author[1]{Youjun Wang} 
\date{}
\begin{document}
\maketitle
\begin{abstract}
	Let $j\geq 2$ be a given integer. Let $f$ be a normalized primitive holomorphic cusp form of even integral weight for the full modular group $\Gamma=SL(2,\mathbb{Z})$. Denote by $\lambda_{\text{sym}^{j}f}(n)$ the $n$th normalized coefficient of the Dirichlet expansion of the $j$th symmetric power $L$-function $L(s,\text{sym}^{j}f)$. In this paper, we are interested in the behavior of the shifted convolution sum involving $\lambda_{\text{sym}^{j}f}(n)$ with a weight function to be the $k$-full kernel function for any fixed integer $k\geq 2$.
\end{abstract}

\textbf{Keywords:} Fourier coefficients, symmetric power $L$-function, cusp form\\

\textbf{2020 Mathematics Subject Classification} 11F11 11F30 11F66
\section{Introduction}
Let $k\geq 2$ be any fixed integer. Then any integer $n\geq 1$ can be uniquely decomposed as $n=q(n)k(n)$, $(q(n),k(n))=1$, where $q(n)$ is $k$-free, and $k(n)$ is $k$-full ($k(n)$ is $k$-full if $p^{k}\mid k(n)$ whenever $p\mid k(n)$). A non-negative integer valued function $a(n)$ is called $k$-full kernel function if $a(n)=a(k(n))$ for all $n\geq 1$ and $a(n)\ll n^{\epsilon}$ for any $\epsilon>0$ by Ivi\'{c} and Tenenbaum \cite{IT}. It is to be noted that $k$-full kernel functions are not necessarily multiplicative.\\
\indent In 1987, Erd{\H{o}}s and Ivi\'{c} \cite{EI} studied the 2-full kernel function and obtained the following asymptotic formulas
\begin{align*}
&\sum_{n\leq x}a(n)d(n+1)=C_{1}x\log x+C_{2}x+O\left(x^{\frac{8}{9}+\epsilon}\right)\\
&\sum_{n\leq x}a(n)\omega(n+1)=D_{1}x\log\log x+D_{2}x+O\left(\frac{x}{\log x}\right),
\end{align*}
where $d(n)$ is the divisor function, and $\omega(n)$ is the number of different prime factors of $n$, and $C_{1}$, $D_{1}>0$, $C_{2}$, $D_{2}$ are constants that can be evaluated explicitly.\\
\indent For an even integer $\kappa\geq 2$, let $H_{\kappa}$ denote the set of all normalized primitive Hecke eigencusp forms of even integral weight $\kappa$ for the full modular group $\Gamma=SL(2,\mathbb{Z})$. For $f(z)\in H_{\kappa}$, it has the Fourier expansion at the cusp $\infty$
\[
f(z)=\sum_{n=1}^{\infty}\lambda_{f}(n)n^{\frac{\kappa-1}{2}}
e^{2\pi inz},\qquad   \Im(z)>0,
\]
where the coefficients $\lambda_{f}(n)$ are eigenvalues of Hecke operators $T_{n}$ with $\lambda_{f}(1)=1$. It's well known that $\lambda_{f}(n)$ is real and satisfies the multiplicative property
\begin{equation}\label{kecheng}
	\lambda_{f}(m)\lambda_{f}(n)=\sum_{d\mid(m,n)}\lambda_{f}
	\left(\frac{mn}{d^{2}}\right),
\end{equation}
for all  $m\geq 1$, $n\geq 1$. In 1974, Deligne \cite{Deligne} proved the Ramanujan--Petersson conjecture for holomorphic cusp forms
\begin{equation}\label{divisor}
	\vert\lambda_{f}(n)\vert\leq d(n),
\end{equation}
where $d(n)$ is the divisor function. By (\ref{divisor}), Deligne's bound is equivalent to the fact that there exist
$\alpha_{f}(p), \beta_{f}(p)\in \mathbb{C}$ satisfying
\begin{align}\label{eq1.1}
	\alpha_{f}(p)+\beta_{f}(p)=\lambda_{f}(p), \qquad
	\vert\alpha_{f}(p)\vert=\vert\beta_{f}(p)\vert=1=\alpha_{f}(p)\beta_{f}(p).
\end{align}
The Hecke $L$-function $L(s,f)$ associated to $f$ is defined as
\begin{align*}
	L(s,f)&=\sum_{n=1}^{\infty}\frac{\lambda_{f}(n)}{n^{s}}=\prod_{p}(1-\lambda_{f}(p)p^{-s}+p^{-2s})^{-1}\\
	&=\prod_{p}\left(1-\frac{\alpha_{f}(p)}{p^{s}}\right)^{-1}\left(1-\frac{\beta_{f}(p)}{p^{s}}\right)^{-1}, \quad \Re(s)>1,
\end{align*}
where $\alpha_{f}(p),\beta_{f}(p)$ are the local parameters satisfying (\ref{eq1.1}). The $j$th \text{sym}metric power $L$-function
associated with $f$ is defined as
\[
L(s,\text{sym}^{j}f)=\prod_{p}\prod_{m=0}^{j}(1-\alpha_{f}^{j-m}(p)\beta_{f}^{m}(p)p^{-s})^{-1},\quad \Re(s)>1.
\]
We may expand it into a Dirichlet series
\begin{align*}
	L(s,\text{sym}^{j}f)&=\sum_{n=1}^{\infty}\frac{\lambda_{\text{sym}^{j}f}(n)}{n^{s}} \notag \\
	&=\prod_{p}\left(1+\frac{\lambda_{\text{sym}^{j}f}(p)}{p^{s}}+\cdots+\frac{\lambda_{\text{sym}^{j}f}(p^{k})}{p^{ks}}+\cdots\right), \quad\Re(s)>1.
\end{align*}
Obviously $\lambda_{\text{sym}^{j}f}(n)$ is a real multiplicative function. For $j=1$, we have $L(s,\text{sym}^{1}f)=L(s,f)$. It is standard that
\[
\lambda_{f}(p^{j})=\lambda_{\text{sym}^{j}f}(p)=\frac{\alpha_{f}^{j+1}(p)-\beta_{f}^{j+1}(p)}{\alpha_{f}(p)-\beta_{f}(p)}=\sum_{m=0}^{j}\alpha_{f}^{j-m}(p)\beta_{f}^{m}(p).
\]
Let $\chi$ be a Dirichlet character modulo $q$. Then we can define the twisted $j$th \text{sym}metric power $L$-function by the Euler product representation with degree $j+1$
\begin{align*}
	L(s,\text{sym}^{j}f\otimes \chi)=&\prod_{p}\prod_{m=0}^{j}(1-\alpha_{f}^{j-m}(p)\beta_{f}^{m}(p)\chi(p)p^{-s})^
	{-1}\\
	=&\sum_{n=1}^{\infty}\frac{\lambda_{\text{sym}^{j}f}(n)\chi(n)}{n^{s}}, \quad \Re(s)>1.
\end{align*}
Shifted convolution sums with $GL(2)$ Fourier coefficients have been investigated extensively by many authors (see \cite{Blomer}, \cite{Holo} and \cite{Michel}), who have obtained many important results, such as equi-distribution quantum unique ergodicity and subconvexity.
In \cite{LvWang}, L\"{u} and Wang investigated the shifted convolution sums of squares of Fourier coefficients with $2$-full kernel function $a(n)$ and obtained an asymptotic formula for the sum
\[
\sum_{n\leq x}a(n)\lambda_{f}^{2}(n+1).
\]
Later, Venkatasubbareddy and Sankaranarayanan \cite{VenSan} considered the problem to higher moments, and generalized to $k$-full kernel functions. In this paper, we are interested in the shifted convolution sum
\[
\sum_{n\leq x}a(n)\lambda_{\textup{sym}^{j}f}^{2}(n+1),
\]
where $a(n)$ is $k$-full kernel functions, for any fixed integer $k\geq 2$.\\

First we state:
\begin{Th}\label{Fouriersum}
	Let $f\in H_{\kappa}$ and $q\geq 100$ be any integer. Then for any $\epsilon>0$ and $q\ll x^{\frac{2}{(j+1)^{2}}-\epsilon}$, we have
	\begin{align*}
		\sum_{n\leq x+1\atop n\equiv 1(q)}\lambda_{\textup{sym}^{j}f}^{2}(n)=c_{f,j}\frac{x}{q}+O\left(\frac{x^{1-\frac{2}{(j+1)^{2}}+\epsilon}q^{1+\epsilon}}{\phi(q)}\right),
	\end{align*}
where $c_{f,j}$ is the constant given by 
\[
c_{f,j}=\prod_{n=1}^{j}L(1,\textup{sym}^{2n}\otimes\chi_{0})H_{j}(1),
\]
and $H_{j}(1)\neq 0$.
\end{Th}
\begin{Th}\label{3ci}
		Let $f\in H_{\kappa}$ and $q\geq 100$ be any integer. Then for any $\epsilon>0$ and $q\ll x^{\frac{46}{617}-\epsilon}$, we have
		\[
		\sum_{n\leq x+1\atop n\equiv 1(q)}\lambda_{\textup{sym}^{2}f}^{3}(n)=\mathcal{C}\frac{x}{q}+O\left(\frac{x^{1-\frac{46}{617}+\epsilon}q^{1+\epsilon}}{\phi(q)}\right),
		\]
		where $\mathcal{C}$ is the constant given by
\end{Th}
\[
\mathcal{C}=L^{3}(1,\textup{sym}^{2}f\otimes\chi_{0})L^{2}(1,\textup{sym}^{4}f\otimes\chi_{0})L(1,\textup{sym}^{6}f\otimes\chi_{0})\mathcal{H}^{(3)}(1),
\]
and $\mathcal{H}^{(3)}(1)\neq 0$.\\

As the applications to Theorem \ref{Fouriersum} and \ref{3ci}, we obtain:
\begin{Th}\label{shiftsum}
	For any integer $k\geq 2$, let $a(n)$ be the $k$-full kernel function and $f\in H_{\kappa}$. Then for any $\epsilon>0$, we have
	\[
	\sum_{n\leq x}a(n)\lambda_{\textup{sym}^{j}f}^{2}(n+1)=C_{f,j}x+O\left(x^{1-\frac{2k-2}{3(j+1)^{2}k}+\epsilon}\right),
	\]
	where $C_{f,j}$ is a constant can be evaluated explicitly.
\end{Th}
\begin{Th}\label{3cishift}
	For any integer $k\geq 2$, let $a(n)$ be the $k$-full kernel function and $f\in H_{\kappa}$. Then for any $\epsilon>0$, we have
	\[
	\sum_{n\leq x}a(n)\lambda_{\textup{sym}^{2}f}^{3}(n+1)=\mathcal{D}x+O\left(x^{\frac{1805k+46}{1851}+\epsilon}\right),
	\]
	where $\mathcal{D}$ is a constant can be evaluated explicitly.
\end{Th}
\textbf{Notation.} Throughout this paper, the letter $\epsilon$ represents a sufficiently small positive constant, not necessarily the same at each occurrences. The constants, both explicit and implicit, in Vinogradov symbols may depend on $f$, $k$, $j$ and $\epsilon$.
\section{Preliminaries lemmas}
In this section, we give some useful lemmas which play important roles in the proofs of the Theorems.
\begin{Lem}\label{fenjie}
	Let $j\geq 2$ be any given integer. Let $f\in H_{\kappa}$ be a Hecker eigenform. Then for $\Re s>1$,  define
	\begin{equation}
		F_{j}(s,\chi):=\sum_{n=1}^{\infty}\frac{\lambda_{\textup{sym}^{j}f}^{2}(n)\chi(n)}{n^{s}}.
	\end{equation}
Then
\[
F_{j}(s,\chi)=G_{j}(s,\chi)H_{j}(s),
\]
where
\[
G_{j}(s,\chi)=L(s,\chi)\prod_{n=1}^{j}L(s,\textup{sym}^{2n}f\otimes\chi).
\]
The function $H_{j}(s)$ admits a Dirichlet series which converges uniformly and absolutely in $\Re s>\frac{1}{2}$ and $H_{j}(1)\neq 0$.
\end{Lem}
\begin{proof}
	Since $\lambda_{\textup{sym}^{j}f}(n)\chi(n)$ is a multiplicative function and satisfies the bound $O(n^{\epsilon})$ for any $\epsilon>0$, then for $\Re s>1$, we have the Euler product
	\begin{equation}\label{product}
		F_{j}(s,\chi)=\prod_{p}\left(1+\sum_{k\geq 1}\frac{\lambda_{\textup{sym}^{j}f}^{2}(p^{k})\chi(p^{k})}{p^{ks}}\right).
	\end{equation} 
Taking $m=n=p^{j}$ in the Hecke relation (\ref{kecheng}), then
\[
\lambda_{f}^{2}(p^{j})=\sum_{d\mid p^{j}}\lambda_{f}\left(\frac{p^{2j}}{d^{2}}\right)=1+\sum_{l=1}^{j}\lambda_{f}(p^{2l}).
\]
Therefore, we have
\begin{align}\label{bp}
	\lambda_{\textup{sym}^{j}f}^{2}(p)\chi(p)&=\lambda_{f}^{2}(p^{j})\chi(p)=\left(1+\sum_{l=1}^{j}\lambda_{f}(p^{2l})\right)\chi(p)  \notag \\
	&=\left(1+\sum_{l=1}^{j}\lambda_{\textup{sym}^{2l}f}(p)\right)\chi(p)=:b(p).
\end{align}
For $\Re s>1$, the $L$-function
\begin{equation}\label{Gbiaoshi}
	G_{j}(s,\chi):=L(s,\chi)\prod_{n=1}^{j}L(s,\textup{sym}^{2n}f\otimes\chi)
\end{equation}
can be represented as 
\begin{equation}\label{Gprod}
	G_{j}(s,\chi)=\prod_{p}\left(1+\sum_{k\geq 1}\frac{b(p^{k})}{p^{ks}}\right).
\end{equation}
By (\ref{product})--(\ref{Gprod}), we obtain
\begin{align*}
	F_{j}(s,\chi)&=G_{j}(s,\chi)\times\prod_{p}\left(1+\frac{\lambda_{\textup{sym}^{j}f}^{2}(p^{2})\chi(p^{2})-b(p^{2})}{p^{2s}}+\cdots\right)\\
	&=:L(s,\chi)\prod_{n=1}^{j}L(s,\textup{sym}^{2n}f\otimes\chi)H_{j}(s).
\end{align*}
Since the bound $O(n^{\epsilon})$ of $\lambda_{\textup{sym}^{j}f}^{2}(n)\chi(n)$ for any $\epsilon>0$, $H_{j}(s)$ converges uniformly and absolutely in the half-plane $\Re s\geq\frac{1}{2}+\epsilon$. 
\end{proof}
\begin{Lem}
	Let $f\in H_{\kappa}$ be a Hecker eigenform. Then for $\Re s>1$, we define
	\[
	\mathcal{F}^{(3)}(s,\chi):=\sum_{n=1}^{\infty}\frac{\lambda_{\textup{sym}^{2}f}^{3}(n)\chi(n)}{n^{s}}.
	\]
	Then we have
	\begin{equation}
		\mathcal{F}^{(3)}(s,\chi)=\mathcal{G}^{(3)}(s,\chi)\mathcal{H}^{(3)}(s),
	\end{equation}
where
\[
\mathcal{G}^{(3)}(s,\chi)=L(s,\chi)L^{3}(s,\textup{sym}^{2}f\otimes\chi)L^{2}(s,\textup{sym}^{4}f\otimes\chi)L(s,\textup{sym}^{6}f\otimes\chi).
\]
The function $\mathcal{H}^{(3)}(s)$ admits a Dirichlet series which converges uniformly and absolutely in $\Re s>\frac{1}{2}$ and $\mathcal{H}^{(3)}(1)\neq 0$.
\end{Lem}
\begin{proof}
	By the definition and (\ref{eq1.1}), we have
	\begin{align*}
		\lambda_{\textup{sym}^{2}f}^{3}(p)=&(\alpha_{f}^{2}(p)+\beta_{f}^{2}(p)+1)^{3}\\
		=&\alpha_{f}^{6}(p)+3\alpha_{f}^{4}(p)+6\alpha_{f}^{2}(p)+7+6\beta_{f}^{2}(p)+3\beta_{f}^{4}(p)+\beta_{f}^{6}(p)\\
		=&\left(\alpha_{f}^{6}(p)+\alpha_{f}^{4}(p)+\alpha_{f}^{2}(p)+1+\beta_{f}^{2}(p)+\beta_{f}^{4}(p)+\beta_{f}^{6}(p)\right)\\
		&+2\left(\alpha_{f}^{4}(p)+\alpha_{f}^{2}(p)+1+\beta_{f}^{2}(p)+\beta_{f}^{4}(p)\right)+3\left(\alpha_{f}^{2}(p)+1+\beta_{f}^{2}(p)\right)+1.
	\end{align*}
Since $\lambda_{\textup{sym}^{2}f}(n)$ and $\chi(n)$ are multiplicative functions, by standard arguments, the above relation leads us to obtain
\begin{align*}
	\mathcal{F}^{(3)}(s,\chi)&=\prod_{p}\left(1+\frac{\lambda_{\textup{sym}^{2}f}^{3}(p)\chi(p)}{p^{s}}+\frac{\lambda_{\textup{sym}^{2}f}^{3}(p^{2})\chi(p^{2})}{p^{2s}}+\cdots\right)\\
	&=L(s,\chi)L^{3}(s,\textup{sym}^{2}f\otimes\chi)L^{2}(s,\textup{sym}^{4}f\otimes\chi)L(s,\textup{sym}^{6}f\otimes\chi)\mathcal{H}^{(3)}(s),
\end{align*}
where $\mathcal{H}^{(3)}(s)$ is some Dirichlet series which converges absolutely and uniformly in $\Re s\geq \frac{1}{2}+\epsilon$ and $\mathcal{H}^{(3)}(1)\neq 0$.
\end{proof}
\begin{Lem}\label{principal}
	Let $\chi_{0}$ be a principal character modulo $q$. Then we have
	\[
	L(s,\chi_{0})=\zeta(s)\prod_{p\mid q}\left(1-\frac{1}{p^{s}}\right)
	\]
	and
	\[
	L(s,\textup{sym}^{2n}f\otimes\chi_{0})=L(s,\textup{sym}^{2n}f)\prod_{p\mid q}\prod_{0\leq j\leq 2n}\left(1-\frac{\alpha_{f}^{2n-2j}(p)}{p^{s}}\right)
	\]
	for all integers $n\geq 1$.
\end{Lem}
\begin{proof}
	See \cite{VenSan} for more details.
\end{proof}
\begin{Lem}\label{induce}
	Let $\chi$ be a non-primitive character modulo $q$ and $\chi^{*}$ be primitive character modulo $q_1(\neq q)$ induced by $\chi$. Then we have
	\[
	L(s,\chi)=L(s,\chi^{*})\prod_{p\mid q\atop p\nmid q_{1}}\left(1-\frac{\chi^{*}(p)}{p^{s}}\right),
	\]
	\[
	L(s,\textup{sym}^{2n}f\otimes\chi)=L(s,\textup{sym}^{2n}f\otimes\chi^{*})\prod_{p\mid q\atop p\nmid q_{1},0\leq j\leq 2n}\left(1-\frac{\alpha_{f}^{2n-2j}(p)\chi^{*}(p)}{p^{s}}\right),
	\]
	\[
	\prod_{p\mid q\atop p\nmid q_{1}}\left(1-\frac{\chi^{*}(p)}{p^{s}}\right)\ll q^{\epsilon} \quad \textup{for} \frac{1}{2}+\epsilon< \Re s<1+\epsilon
	\]
	and
	\[
	\prod_{p\mid q\atop p\nmid q_{1},0\leq j\leq 2n}\left(1-\frac{\alpha_{f}^{2n-2j}(p)\chi^{*}(p)}{p^{s}}\right)\ll q^{\epsilon} \quad \textup{for} \frac{1}{2}+\epsilon< \Re s<1+\epsilon.
	\]
	for all integers $n\geq 1$.
\end{Lem}
\begin{proof}
	See \cite{VenSan}.
\end{proof}
\begin{Lem}\label{zetashuiping}
	For $\frac{1}{2}\leq \sigma\leq 2$, $T$-sufficiently large, there exists a $T^{*}\in[T,T+T^{1/3}]$ such that the bound
	\[
	\log\zeta(\sigma+iT^{*})\ll (\log\log T^{*})^{2}\ll (\log\log T)^{2}
	\]
	holds uniformly for $\frac{1}{2}\leq \sigma\leq 2$ and thus we have
	\[
	\vert \zeta(\sigma+iT^{*})\vert\ll \exp((\log\log T^{*})^{2})\ll_{\epsilon}T^{\epsilon}
	\]
	on the horizontal line with $t=T^{*}$ uniformly for $\frac{1}{2}\leq \sigma\leq 2$.
\end{Lem}
\begin{proof}
	See \cite[Lemma 1]{RamSan}.
\end{proof}
\begin{Lem}\label{zeta}
	If $A\geq 4$ is a fixed number, then for $4\leq A\leq 12$, we have
	\begin{equation}\label{zetamoment}
		\int_{1}^{T}\vert\zeta(\frac{1}{2}+it)\vert^{A}\textup{d}t\ll T^{1+\frac{A-4}{8}+\epsilon}
	\end{equation}
holds uniformly; and for $\frac{1}{2}\leq \sigma\leq 2+\epsilon$
\begin{equation}\label{zetabound}
	\zeta(\sigma+it)\ll (1+\vert t\vert)^{\max\{\frac{13}{42}(1-\sigma),0\}+\epsilon}
\end{equation}
\end{Lem}
\begin{proof}
	For (\ref{zetamoment}), see \cite[Theorem 8.3]{Ivic}, and for (\ref{zetabound}), see \cite{Bourgain}.
\end{proof}
\begin{Lem}\label{GL1}
	Let $\chi$ be a primitive character modulo $q$. Then for $q\ll T^{2}$, we have
	\begin{equation}\label{GL1bound}
		L(\sigma+iT,\chi)\ll (q(1+\vert T\vert))^{\max\{\frac{1}{3}(1-\sigma),0\}+\epsilon}
	\end{equation}
holds uniformly for $\frac{1}{2}\leq\sigma\leq 2$; and
\begin{equation}\label{GL1moment}
	\int_{1}^{T}\vert L(\sigma+it,\chi)\vert^{4}\textup{d}t\ll (qT)^{2(1-\sigma)+\epsilon}
\end{equation}
uniformly for $\frac{1}{2}\leq \sigma\leq 1+\epsilon$ and $T\geq 1$.
\end{Lem}
\begin{proof}
	The first result follows from D. R. Heath-Brown \cite{Heath}, and the second result follows from Perelli \cite{Perelli}.
\end{proof}
\begin{Lem}\label{GL3}
	Let $f\in H_{k}$ and $\chi$ be a primitive character modulo $q$. Then for $q\ll T^{2}$, we have
	\begin{equation}\label{GL3bound}
	L(\sigma+iT,\textup{sym}^2 f)\ll (1+\vert T\vert)^{\max\{\frac{6}{5}(1-\sigma),0\}+\epsilon}
	\end{equation}
and
\begin{equation}\label{GL3qbound}
	L(\sigma+iT,\textup{sym}^2 f\otimes\chi)\ll (q(1+\vert T\vert))^{\max\{\frac{67}{46}(1-\sigma),0\}+\epsilon}
\end{equation}
	uniformly for $\frac{1}{2}\leq \sigma\leq 2$ and $\vert T\vert\geq 1$;
	\begin{equation}\label{GL3qmoment}
		\int_{1}^{T}\vert L(\sigma+iT,\textup{sym}^2 f\otimes\chi)\vert^{4}\textup{d}t\ll (qT)^{6(1-\sigma)+\epsilon}
	\end{equation}
uniformly for $\frac{1}{2}\leq \sigma\leq 1+\epsilon$ and $T\geq 1$.
\end{Lem}
\begin{proof}
	(\ref{GL3bound}) and (\ref{GL3qbound}) follow from Phragm\'{e}n-Lindel\"{o}f convexity principle and the work of Lin, Nunes and Qi \cite{lin} and Huang \cite{Huang} respectively. (\ref{GL3qmoment}) follows from Perelli \cite{Perelli}.   
\end{proof}
\begin{Lem}\label{GLn}
	Let $\chi$ be a primitive character modulo $q$. For the general $L$-functions $\mathcal{L}_{m,n}^{d}(s,\chi)$ with degree $2A$, for $1/2\leq\sigma\leq 1$ and any $\epsilon>0$, we have
	\begin{align}\label{GLnmoment}
		&\int_{0}^{T}\vert \mathcal{L}_{m,n}^{d}(\sigma+it,\chi)\vert^2\textup{d}t\ll T(qT)^{\omega(1/2)\epsilon}+(qT)^{2A(1-\sigma)+\epsilon},
	\end{align}
	where $\omega(1/2)=1$ or $0$ according as $\sigma=1/2$ or $1/2<\sigma\leq 1$, respectively. Furthermore,
	\begin{equation}\label{GLnbound}
		\mathcal{L}_{m,n}^{d}(\sigma+it,\chi)\ll (q\vert t\vert+1)^{\max\{A(1-\sigma),0\}+\epsilon},
	\end{equation}
	uniformly for $\epsilon \leq \sigma \leq 1+\epsilon$.
\end{Lem}
\begin{proof}
	(\ref{GLnmoment}) is due to Mastsumoto \cite[Section 3]{Matsumoto}, which is a supplement of Perelli \cite[Thoerem 4]{Perelli}. (\ref{GLnbound}) is given by Jiang and L\"{u} \cite[Lemma 2.4]{Jiang}.
\end{proof}
\section{Proof of Theorem \ref{Fouriersum}}
Let $\chi$ be any Dirichlet character modulo $q$. By the orthogonality we get
\begin{align*}
	\sum_{n\leq x+1\atop n\equiv 1(q)}\lambda_{\textup{sym}^{j}f}^{2}(n)=&\frac{1}{\phi(q)}\sum_{\chi(q)}\sum_{n\leq x+1}\lambda_{\textup{sym}^{j}f}^{2}(n)\chi(n)\\
	=&\frac{1}{\phi(q)}\sum_{n\leq x+1}\lambda_{\textup{sym}^{j}f}^{2}(n)\chi_{0}(n)+\frac{1}{\phi(q)}\sum_{n\leq x+1}\sum_{\chi\neq \chi_{0}\atop \chi \textup{ is primitvie}}\lambda_{\textup{sym}^{j}f}^{2}(n)\chi(n)\\
	&+\frac{1}{\phi(q)}\sum_{n\leq x+1}\sum_{\chi\neq \chi_{0}\atop \chi \textup{ is non-primitvie}}\lambda_{\textup{sym}^{j}f}^{2}(n)\chi(n)\\
	=&:\frac{1}{\phi(q)}\big(\textstyle{\sum_{1}+\sum_{2}+\sum_{3}}\big).
\end{align*}
For $\sum_{1}$, by Lemma \ref{fenjie} and applying Perron's formula, we have
\begin{align*}
	\sum_{n\leq x+1}\lambda_{\textup{sym}^{j}f}^{2}(n)\chi_{0}(n)=&\frac{1}{2\pi i}\int_{1+\epsilon-iT}^{1+\epsilon+iT}F_{j}(s,\chi_{0})\frac{(x+1)^{s}}{s}\textup{d}s+O\left(\frac{x^{1+\epsilon}}{T}\right).
\end{align*}
Shifting the line of integration to $\Re s=\eta:=\frac{1}{2}+\epsilon$ and using Cauchy's residue theorem, we obtain
\begin{align*}
	\sum_{n\leq x+1}\lambda_{\textup{sym}^{j}f}^{2}(n)\chi_{0}(n)=&\Res_{s=1}F_{j}(s,\chi_{0})\frac{(x+1)^{s}}{s}\\
	&+\frac{1}{2\pi i}\left\{\int_{\eta-iT}^{\eta+iT}+\int_{\eta+iT}^{1+iT}+\int_{1-iT}^{\eta-iT}\right\}F_{j}(s,\chi_{0})\frac{(x+1)^{s}}{s}\textup{d}s+O\left(\frac{x^{1+\epsilon}}{T}\right)\\
	=&:c_{f,j}\frac{\phi(q)}{q}x+I_{1}+I_{2}+I_{3}+O\left(\frac{x^{1+\epsilon}}{T}\right).
\end{align*}
Here 
\[
c_{f,j}=\prod_{n=1}^{j}L(1,\text{sym}^{2n}\otimes\chi_{0})H_{j}(1).
\]
For $I_{1}$, by (\ref{zetamoment}), (\ref{GL3bound}) and (\ref{GLnmoment}), along with  H\"{o}lder inequality, we obtain
\begin{align*}
	I_{1}=&\frac{1}{2\pi i}\int_{\eta-iT}^{\eta+iT}\zeta(s)\prod_{p\mid q}\left(1-\frac{1}{p^{s}}\right)L(s,\textup{sym}^{2}f\otimes\chi_{0})\prod_{n=2}^{j}L(s,\textup{sym}^{2n}f\otimes\chi_{0})H_{j}(s)\frac{(x+1)^{s}}{s}\textup{d}s\\
	\ll&\int_{-T}^{T}\Big\vert\zeta(\frac{1}{2}+\epsilon+it)L(\frac{1}{2}+\epsilon+it,\textup{sym}^{2}f\otimes\chi_{0})\prod_{n=2}^{j}L(\frac{1}{2}+\epsilon+it,\textup{sym}^{2n}f\otimes\chi_{0})\frac{x^{\frac{1}{2}+\epsilon}}{t}\Big\vert\textup{d}s\\
	\ll& x^{\frac{1}{2}+\epsilon}\log T\max_{1\leq T_{1}\leq T}\Bigg\{\frac{1}{T_{1}}\left(\int_{T_1}^{2T_{1}}\vert\zeta(\frac{1}{2}+\epsilon+it)\vert^{12}\textup{d}t\right)^{1/12}\left(\int_{T_1}^{2T_{1}}\vert L(\frac{1}{2}+\epsilon+it,\textup{sym}^{2}f)\vert^{12/5}\textup{d}t\right)^{5/12}\\
	&\times\left(\int_{T_1}^{2T_{1}}\vert\prod_{n=2}^{j}L(\frac{1}{2}+\epsilon+it,\textup{sym}^{2n}f)\vert^{2}\textup{d}t\right)^{1/2}\Bigg\}+x^{\frac{1}{2}+\epsilon}\\
	\ll&x^{\frac{1}{2}+\epsilon}T^{2\cdot\frac{1}{12}+(\frac{2}{5}\cdot\frac{6}{5}\cdot\frac{1}{2}+\frac{3}{2})\cdot\frac{5}{12}+((j+1)^{2}-4)\cdot\frac{1}{2}\cdot\frac{1}{2}-1+\epsilon}\\
	\ll&x^{\frac{1}{2}+\epsilon}T^{\frac{1}{4}(j+1)^{2}-\frac{133}{120}+\epsilon}.
\end{align*}
For the integrals over the horizontal segments $I_{2}$ and $I_{3}$, by (\ref{zetashuiping}), (\ref{GL3bound}) and (\ref{GLnbound}), we have
\begin{align*}
	I_{2}+I_{3}&\ll \int_{1/2+\epsilon}^{1+\epsilon}x^{\sigma}\Big\vert \zeta(\sigma+it)\prod_{n=1}^{j}L(\text{sym}^{2n}f,\sigma+it)\Big\vert
	T^{-1}\textup{d}\sigma\notag\\
	&\ll\max_{1/2+\epsilon\leq \sigma\leq1+\epsilon}x^{\sigma}T^{(\frac{6}{5}+\frac{1}{2}((j+1)^{2}-4))
		(1-\sigma)+\epsilon}
	\notag T^{-1}\\
	&\ll \frac{x^{1+\epsilon}}{T}+x^{\frac{1}{2}+\epsilon}T^{\frac{1}{4}(j+1)^2-\frac{7}{5}+\epsilon}.
\end{align*}
Hence
\begin{align*}
	\sum_{n\leq x+1}\lambda_{\textup{sym}^{j}f}^{2}(n)\chi_{0}(n)=c_{f,j}\frac{\phi(q)}{q}x+O\left(x^{\frac{1}{2}+\epsilon}T^{\frac{1}{4}(j+1)^{2}-\frac{133}{120}+\epsilon}\right)+O\left(\frac{x^{1+\epsilon}}{T}\right).
\end{align*}
Now for $\sum_{2}$ and $\sum_{3}$, by Lemma \ref{induce}, we just need to consider the sum $\sum_{2}$. By Perron's formula we get
\begin{align*}
	\sum_{n\leq x+1,\chi\neq \chi_{0}\atop \chi \textup{ is primitvie}}\lambda_{\textup{sym}^{j}f}^{2}(n)\chi(n)=\frac{1}{2\pi i}\int_{1+\epsilon-iT}^{1+\epsilon+iT}F_{j}(s,\chi)\frac{(x+1)^{s}}{s}\textup{d}s+O\left(\frac{x^{1+\epsilon}}{T}\right).
\end{align*}
Moving the line of integration to $\Re s=\eta=:\frac{1}{2}+\epsilon$ and using Cauchy's residue theorem, we have
\begin{align*}
	\sum_{n\leq x+1,\chi\neq \chi_{0}\atop \chi \textup{ is primitvie}}\lambda_{\textup{sym}^{j}f}^{2}(n)\chi(n)&=\frac{1}{2\pi i}\left\{\int_{\eta-iT}^{\eta+iT}+\int_{\eta+iT}^{1+iT}+\int_{1-iT}^{\eta-iT}\right\}F_{j}(s,\chi)\frac{(x+1)^{s}}{s}\textup{d}s+O\left(\frac{x^{1+\epsilon}}{T}\right)\\
	&=:I_{4}+I_{5}+I_{6}+O\left(\frac{x^{1+\epsilon}}{T}\right).
\end{align*}
For the vertical line, we have
\begin{align*}
	I_{4}\ll& \int_{-T}^{T}\Big\vert L(\frac{1}{2}+\epsilon+it,\chi)L(\frac{1}{2}+\epsilon+it,\text{sym}^{2}f\otimes\chi)\prod_{n=2}^{j}L(\frac{1}{2}+\epsilon+it,\text{sym}^{2n}f\otimes\chi)\frac{x^{\frac{1}{2}+\epsilon}}{t}\Big\vert\textup{d}t\\
	\ll& x^{\frac{1}{2}+\epsilon}\max_{1\leq T_{1}\leq T}\Bigg\{ T_{1}^{-1}\left(\int_{T_{1}}^{2T_{1}}\Big\vert L(\frac{1}{2}+\epsilon+it,\chi)\Big\vert^{4}\textup{d}t\right)^{\frac{1}{4}}\left(\int_{T_{1}}^{2T_{1}}\Big\vert L(\frac{1}{2}+\epsilon+it,\text{sym}^{2}f\otimes\chi)\Big\vert^{4}\textup{d}t\right)^{\frac{1}{4}}\\
	&\times\left(\int_{T_{1}}^{2T_{1}}\Big\vert \prod_{n=2}^{j}L(\frac{1}{2}+\epsilon+it,\text{sym}^{2n}f\otimes\chi)\Big\vert^{2}\textup{d}t\right)^{\frac{1}{2}}\Bigg\}+x^{\frac{1}{2}+\epsilon}\\
	\ll& x^{\frac{1}{2}+\epsilon}(qT)^{\frac{1}{4}+\epsilon}(qT)^{\frac{3}{4}+\epsilon}(qT)^{\frac{1}{4}((j+1)^{2}-4)+\epsilon}T^{-1}\\
	\ll&x^{\frac{1}{2}+\epsilon}q^{\frac{1}{4}(j+1)^{2}+\epsilon}T^{\frac{1}{4}(j+1)^{2}-1+\epsilon}.
\end{align*}
For the horizontal segments $I_{5}$ and $I_{6}$, we have
\begin{align*}
	I_{5}+I_{6}&\ll \int_{\frac{1}{2}+\epsilon}^{1+\epsilon}\big\vert L(\sigma+iT,\chi)L(\sigma+iT,\text{sym}^{2}f\otimes\chi)\prod_{n=2}^{j}L(\sigma+iT,\text{sym}^{2n}f\otimes\chi)\big\vert x^{\sigma}T^{-1}\textup{d}\sigma\\
	&\ll \max_{1/2+\epsilon\leq \sigma\leq1+\epsilon}x^{\sigma}(qT)^{(\frac{1}{3}+\frac{67}{46}+\frac{1}{2}((j+1)^{2}-4))(1-\sigma)+\epsilon}T^{-1}\\
	&\ll\max_{1/2+\epsilon\leq \sigma\leq1+\epsilon}x^{\sigma}(qT)^{(\frac{1}{2}(j+1)^{2}-\frac{29}{138})(1-\sigma)+\epsilon}T^{-1}\\
	&\ll\frac{x^{1+\epsilon}}{T}+x^{\frac{1}{2}+\epsilon}q^{\frac{1}{4}(j+1)^{2}-\frac{29}{276}+\epsilon}T^{\frac{1}{4}(j+1)^{2}-\frac{305}{276}+\epsilon}.
\end{align*}
Therefore, we get
\begin{align*}
	\sum_{n\leq x+1,\chi\neq \chi_{0}\atop \chi \textup{ is primitvie}}\lambda_{\textup{sym}^{j}f}^{2}(n)\chi(n)=O\left(\frac{x^{1+\epsilon}}{T}\right)+O(x^{\frac{1}{2}+\epsilon}q^{\frac{1}{4}(j+1)^{2}+\epsilon}T^{\frac{1}{4}(j+1)^{2}-1+\epsilon}).
\end{align*}
Hence, we have
\begin{align*}
	\sum_{n\leq x+1\atop n\equiv 1(q)}\lambda_{\textup{sym}^{j}f}^{2}(n)=c_{f,j}\frac{x}{q}+O\left(\frac{x^{1+\epsilon}}{T\phi(q)}\right)+O(\frac{x^{\frac{1}{2}+\epsilon}q^{\frac{1}{4}(j+1)^{2}+\epsilon}T^{\frac{1}{4}(j+1)^{2}-1+\epsilon}}{\phi(q)}).
\end{align*}
Choosing $T=\frac{x^{\frac{2}{(j+1)^{2}}}}{q}$, we obtain
\[
\sum_{n\leq x+1\atop n\equiv 1(q)}\lambda_{\textup{sym}^{j}f}^{2}(n)=c_{f,j}\frac{x}{q}+O\left(\frac{x^{1-\frac{2}{(j+1)^{2}}+\epsilon}q^{1+\epsilon}}{\phi(q)}\right),
\]
for $q\ll x^{\frac{2}{(j+1)^{2}}-\epsilon}$.
\section{Proof of Theorem \ref{3ci}}
Following the similar arguments as in the proof of Theorem \ref{Fouriersum}, we obtain
\[
\sum_{n\leq x+1\atop n\equiv 1(q)}\lambda_{\textup{sym}^{2}f}^{3}(n)=\mathcal{C}\frac{x}{q}+O\left(\frac{x^{1+\epsilon}}{\phi(q)T}\right)+O\left(\frac{x^{\frac{1}{2}+\epsilon}q^{\frac{617}{92}+\epsilon}T^{\frac{525}{92}+\epsilon}}{\phi(q)}\right).
\]
Taking $T=\frac{x^{\frac{46}{617}}}{q}$, we get
\[
\sum_{n\leq x+1\atop n\equiv 1(q)}\lambda_{\textup{sym}^{2}f}^{3}(n)=\mathcal{C}\frac{x}{q}+O\left(\frac{x^{1-\frac{46}{617}+\epsilon}q^{1+\epsilon}}{\phi(q)}\right),
\]
for $q\ll x^{\frac{46}{617}-\epsilon}$.
\section{Proof of Theorem \ref{shiftsum}}
By the definition of $k$-full kernel function $a(n)$, when we decompose $n\geq 1$ uniquely as $n=q(n)k(n)$, $(k(n),q(n))=1$, where $k(n)$ is $k$-full and $q(n)$ is $k$-free. There holds
\[
a(n)=a(k(n))\ll n^{\epsilon}, \qquad \lambda_{\textup{sym}^{j}f}(n)\ll n^{\epsilon}.
\]
Let $1\leq H\ll x^{\frac{2}{(j+1)^{2}}}$, then
\begin{align}\label{eqfj1}
	&\sum_{n\leq x}a(n)\lambda_{\textup{sym}^{j}f}^{2}(n+1)\notag\\
	&=\sum_{n\leq x\atop k(n)\leq H}a(n)\lambda_{\textup{sym}^{j}f}^{2}(n+1)+\sum_{n\leq x\atop k(n)> H}a(n)\lambda_{\textup{sym}^{j}f}^{2}(n+1)\notag\\
	&=\sum_{n\leq x\atop k(n)\leq H}a(n)\lambda_{\textup{sym}^{j}f}^{2}(n+1)+O\left(\sum_{H<k(n)\leq x}a(k(n))\sum_{q(n)\leq\frac{x}{k(n)}\atop(q(n),k(n))=1}\lambda_{\textup{sym}^{j}f}^{2}(k(n)q(n)+1)\right)\notag\\
	&=\sum_{n\leq x\atop k(n)\leq H}a(n)\lambda_{\textup{sym}^{j}f}^{2}(n+1)+O\left(\sum_{H<k(n)\leq x}(k(n))^{\epsilon}\sum_{q(n)\leq\frac{x}{k(n)}\atop(q(n),k(n))=1}(k(n)q(n))^{\epsilon}\right)\notag\\
	&=\sum_{n\leq x\atop k(n)\leq H}a(n)\lambda_{\textup{sym}^{j}f}^{2}(n+1)+O\left(\sum_{H<k(n)\leq x}(k(n))^{\epsilon}(\frac{x}{k(n)})^{1+\epsilon}\right)\notag\\
	&=\sum_{n\leq x\atop k(n)\leq H}a(n)\lambda_{\textup{sym}^{j}f}^{2}(n+1)+O\left(x^{1+\epsilon}\sum_{H<k(n)\leq x}\frac{1}{k(n)}\right)\notag\\
	&=\sum_{n\leq x\atop k(n)\leq H}a(n)\lambda_{\textup{sym}^{j}f}^{2}(n+1)+O(x^{1+\epsilon}H^{\frac{1}{k}-1}).
\end{align}
Define $g(l)=\sum\limits_{md^{k}=l}\mu(d)$, then $g(q(n))=1$. We have
\begin{align}\label{eqfj*}
	\sum_{n\leq x\atop k(n)\leq H}a(n)\lambda_{\textup{sym}^{j}f}^{2}(n+1)=&\sum_{k(n)\leq H}a(k(n))\sum_{q(n)\leq\frac{x}{k(n)}\atop(q(n),k(n))=1}\lambda_{\textup{sym}^{j}f}^{2}(k(n)q(n)+1)\notag\\
	=&\sum_{k(n)\leq H}a(k(n))\sum_{q(n)\leq\frac{x}{k(n)}\atop(q(n),k(n))=1}g(q(n))\lambda_{\textup{sym}^{j}f}^{2}(k(n)q(n)+1)\notag\\
	=&\sum_{k(n)\leq H}a(k(n))\sum_{q(n)\leq\frac{x}{k(n)}\atop(q(n),k(n))=1}\sum_{m(n)d^{k}(n)=q(n)}\mu(d(n))\lambda_{\textup{sym}^{j}f}^{2}(k(n)q(n)+1)\notag\\
	=&\sum_{k(n)\leq H}a(k(n))\sum_{d(n)\leq(\frac{x}{k(n)})^{\frac{1}{k}}\atop(d(n),k(n))=1}\mu(d(n))\notag\\
	&\times\sum_{m(n)\leq\frac{x}{k(n)d^{k}(n)}\atop(m(n),k(n))=1}\lambda_{\textup{sym}^{j}f}^{2}(k(n)m(n)d^{k}(n)+1)\notag\\
	=&:\textstyle{\sum_{1}^{*}+\sum_{2}^{*}},
\end{align} 
where
\begin{align*}
	{\textstyle\sum_{1}^{*}}=\sum_{k(n)\leq H}a(k(n))\sum_{d(n)\leq H^{\frac{1}{k}}\atop(d(n),k(n))=1}\mu(d(n))\sum_{m(n)\leq\frac{x}{k(n)d^{k}(n)}\atop(m(n),k(n))=1}\lambda_{\textup{sym}^{j}f}^{2}(k(n)m(n)d^{k}(n)+1)
\end{align*}
and
\begin{align*}
	{\textstyle\sum_{2}^{*}}=\sum_{k(n)\leq H}a(k(n))\sum_{H^{\frac{1}{k}}<d(n)\leq(\frac{x}{k(n)})^{\frac{1}{k}}\atop(d(n),k(n))=1}\mu(d(n))\sum_{m(n)\leq\frac{x}{k(n)d^{k}(n)}\atop(m(n),k(n))=1}\lambda_{\textup{sym}^{j}f}^{2}(k(n)m(n)d^{k}(n)+1).
\end{align*}
For $\sum_{2}^{*}$, we have
\begin{align}\label{2*bound}
	{\textstyle\sum_{2}^{*}}&\ll \sum_{k(n)\leq H}(k(n))^{\epsilon}\sum_{d(n)\geq H^{\frac{1}{k}}}\sum_{m(n)\leq\frac{x}{k(n)d^{k}(n)}\atop(m(n),k(n))=1}(k(n)m(n)d^{k}(n))^{\epsilon}\notag\\
	&\ll \sum_{k(n)\leq H}(k(n))^{2\epsilon}\sum_{d(n)\geq H^{\frac{1}{k}}}(d(n))^{k\epsilon}\left(\frac{x}{k(n)d^{k}(n)}\right)^{1+\epsilon}\notag\\
	&\ll x^{1+\epsilon}H^{\frac{1}{k}-1}.
\end{align}
For $\sum_{1}^{*}$, by using $\sum\limits_{\delta(n)\mid(m(n),k(n))}\mu(\delta(n))=1$, we obtain
\begin{align*}
	{\textstyle\sum_{1}^{*}}=&\sum_{k(n)\leq H}a(k(n))\sum_{d(n)\leq H^{\frac{1}{k}}\atop(d(n),k(n))=1}\mu(d(n))\sum_{\delta(n)\mid k(n)}\mu(\delta(n))\\
	&\times\sum_{m_{1}(n)\delta(n)k(n)d^{k}(n)\leq x}\lambda_{\textup{sym}^{j}f}^{2}(k(n)m_{1}(n)\delta(n)d^{k}(n)+1).
\end{align*}
Write
\begin{align*}
	\sum_{m_{1}(n)\delta(n)k(n)d^{k}(n)\leq x}\lambda_{\textup{sym}^{j}f}^{2}(k(n)m_{1}(n)\delta(n)d^{k}(n)+1)=:\sum_{n\leq x+1\atop n\equiv1(k(n)\delta(n)d^{k}(n))}\lambda_{\textup{sym}^{j}f}^{2}(n).
\end{align*}
By Theorem \ref{Fouriersum}, we get $\sum_{1}^{*}=\sum_{1}^{\prime}+\sum_{1}^{\prime\prime}$, where
\begin{align*}
	{\textstyle\sum_{1}^{\prime}}=&\sum_{k(n)\leq H}a(k(n))\sum_{d(n)\leq H^{\frac{1}{k}}\atop(d(n),k(n))=1}\mu(d(n))\sum_{\delta(n)\mid k(n)}\mu(\delta(n))\times c_{f,j}x\frac{1}{(k(n)\delta(n)d^{k}(n))}
\end{align*}
and
\begin{align*}
	{\textstyle\sum_{1}^{\prime\prime}}=O\Bigg(\sum_{k(n)\leq H}a(k(n))&\sum_{d(n)\leq H^{\frac{1}{k}}\atop(d(n),k(n))=1}\mu(d(n))\sum_{\delta(n)\mid k(n)}\mu(\delta(n))\times x^{1-\frac{2}{(j+1)^{2}}+\epsilon}\frac{(k(n)\delta(n)d^{k}(n))^{1+\epsilon}}{\phi(k(n)\delta(n)d^{k}(n))}\Bigg)
\end{align*}
For $\sum_{1}^{\prime\prime}$, we have
\begin{align}\label{1ppbound}
	{\textstyle\sum_{1}^{\prime\prime}}&=O\Bigg(\sum_{k(n)\leq H}a(k(n))\sum_{d(n)\leq H^{\frac{1}{k}}\atop(d(n),k(n))=1}\mu(d(n))\sum_{\delta(n)\mid k(n)}\mu(\delta(n))\times x^{1-\frac{2}{(j+1)^{2}}+\epsilon}\frac{(k(n)\delta(n))^{1+\epsilon}(d^{k}(n))^{1+\epsilon}}{\phi(k(n)\delta(n))\phi(d^{k}(n))}\Bigg)\notag\\
	&\ll x^{1-\frac{2}{(j+1)^{2}}+\epsilon}\sum_{k(n)\leq H}(k(n))^{\epsilon}\sum_{d(n)\leq H^{\frac{1}{k}}\atop(d(n),k(n))=1}\sum_{\delta(n)\mid k(n)}\log\log(k(n)\delta(n))\log\log(d^{k}(n))\notag\\
	&\ll x^{1-\frac{2}{(j+1)^{2}}+\epsilon}\log\log H\log\log H^{2}\sum_{k(n)\leq H}(k(n))^{2\epsilon}H^{\frac{1}{k}}\notag\\
	&\ll x^{1-\frac{2}{(j+1)^{2}}+4\epsilon}H^{\frac{2}{k}}.
\end{align}
Now for $\sum_{1}^{\prime}$, we have
\begin{align}\label{1pbound}
	{\textstyle\sum_{1}^{\prime}}=&c_{f,j}x\sum_{k(n)\leq H}a(k(n))\sum_{d(n)\leq H^{\frac{1}{k}}\atop(d(n),k(n))=1}\mu(d(n))\sum_{\delta(n)\mid k(n)}\mu(d(n))\times\frac{1}{(k(n)\delta(n)d^{k}(n))}\notag\\
	=&c_{f,j}x\sum_{k(n)\leq H}\frac{a(k(n))}{k(n)}\sum_{d(n)\leq H^{\frac{1}{k}}\atop(d(n),k(n))=1}\frac{\mu(d(n))}{d^{k}(n)}\sum_{\delta(n)\mid k(n)}\frac{\mu(\delta(n))}{\delta(n)}\notag\\
	=&c_{f,j}x\sum_{k(n)=1}^{\infty}\frac{a(k(n))}{k(n)}\sum_{\delta(n)\mid k(n)}\frac{\mu(\delta(n))}{\delta(n)}\sum_{d(n)\leq H^{\frac{1}{k}}\atop(d(n),k(n))=1}\frac{\mu(d(n))}{d^{k}(n)}\\
	&+O\Bigg(c_{f,j}x\sum_{k(n)> H}\frac{a(k(n))}{k(n)}\sum_{\delta(n)\mid k(n)}\frac{\mu(\delta(n))}{\delta(n)}\sum_{d(n)\leq H^{\frac{1}{k}}\atop(d(n),k(n))=1}\frac{\mu(d(n))}{d^{k}(n)}
	\Bigg)\notag\\
	=&c_{f,j}x\sum_{k(n)=1}^{\infty}\frac{a(k(n))}{k(n)}\sum_{\delta(n)\mid k(n)}\frac{\mu(\delta(n))}{\delta(n)}\sum_{d(n)=1\atop(d(n),k(n))=1}^{\infty}\frac{\mu(d(n))}{d^{k}(n)}
	\notag\\
	&+O\Bigg(c_{f,j}x\sum_{k(n)> H}\frac{a(k(n))}{k(n)}\sum_{\delta(n)\mid k(n)}\frac{\mu(\delta(n))}{\delta(n)}\sum_{d(n)\leq H^{\frac{1}{k}}\atop(d(n),k(n))=1}\frac{\mu(d(n))}{d^{k}(n)}\Bigg)\notag\\
	&+O\Bigg(c_{f,j}x\sum_{k(n)=1}^{\infty}\frac{a(k(n))}{k(n)}\sum_{\delta(n)\mid k(n)}\frac{\mu(\delta(n))}{\delta(n)}\sum_{d(n)> H^{\frac{1}{k}}}\frac{\mu(d(n))}{d^{k}(n)}\Bigg)\notag\\
	=&:C_{f,j}x+O\left(x^{1+\epsilon}H^{\frac{1}{k}-1}\right).
\end{align}
From (\ref{eqfj1}), (\ref{eqfj*}), (\ref{2*bound}), (\ref{1ppbound}) and (\ref{1pbound}), we have
\begin{align*}
	\sum_{n\leq x}a(n)\lambda_{\textup{sym}^{j}f}^{2}(n+1)=C_{f,j}x+O\left(x^{1+\epsilon}H^{\frac{1}{k}-1}\right)+O\left(x^{1-\frac{2}{(j+1)^{2}}+\epsilon}H^{\frac{2}{k}}\right).
\end{align*}
Note that by Theorem \ref{Fouriersum}, we have $q=\delta(n)k(n)d^{k}(n)$, and $\delta(n)\leq H$, $k(n)\leq H$ and $d^{k}(n)\leq H$. Thus we get $q\leq H^{3}$, furthermore, by Theorem \ref{Fouriersum}, $q$ has to satisfy the condition $q\ll x^{\frac{2}{(j+1)^{2}}-\epsilon}$. This leads us to choose $H$ to be the optimal possible value satisfying $q\leq H^{3}$ and $q\ll x^{\frac{2}{(j+1)^{2}}-\epsilon}$. Thus we choose $H=x^{\frac{2}{3(j+1)^{2}}-\epsilon}$ and we obtain
\begin{align*}
	\sum_{n\leq x}a(n)\lambda_{\textup{sym}^{j}f}^{2}(n+1)=C_{f,j}x+O\left(x^{1-\frac{2k-2}{3(j+1)^{2}k}+\epsilon}\right).
\end{align*}
\section{Proof of Theorem \ref{3cishift}}
Following the similar arguments as in the proof of Theorem \ref{shiftsum}, along with Theorem \ref{3ci}, we get
\[
\sum_{n\leq x}a(n)\lambda_{\textup{sym}^{j}f}^{2}(n+1)=\mathcal{D}x+O\left(x^{1+\epsilon}H^{\frac{1}{k}-1}\right)+O\left(x^{1-\frac{46}{617}+\epsilon}H^{\frac{2}{k}}\right),
\]
and by choosing $H=x^{\frac{46}{1851}-\epsilon}$, we have
\[
\sum_{n\leq x}a(n)\lambda_{\textup{sym}^{2}f}^{3}(n+1)=\mathcal{D}x+O\left(x^{\frac{1805k+46}{1851}+\epsilon}\right).
\]
\section*{Acknowledgments} 
The author would like to thank Prof. Hengcai Tang for his guidance and also anonymous referee for some fruitful comments.

\halign{\small\qquad\quad#\hfill\qquad&\small#\hfill\qquad&\small#\hfill \qquad &
	\small#\hfill \cr
	Youjun Wang                                              \cr
	School of Mathematics and Statistics            \cr
	Henan University                                          \cr
	Kaifeng, Henan 475004                                     \cr
	P. R. China                                               \cr
	\texttt{math$\_$wyj@henu.edu.cn}                               \cr
}
\end{document}